\documentclass[a4paper]{jpconf}

\usepackage{iopams,amsmath,amsthm}
\newtheorem{thm}{Theorem}

\newcommand{\C}{{\mathbb{C}}}

\newcommand{\N}{{\mathbb{N}}}

\begin{document}
\title{Substitution-based sequences with absolutely continuous diffraction}

\author{Lax Chan and Uwe Grimm}

\address{School of Mathematics and Statistics, The Open University, 
Milton Keynes MK7 6AA, UK}
\ead{lax.chan@open.ac.uk, uwe.grimm@open.ac.uk}

\begin{abstract}
  Modifying Rudin's original construction of the Rudin--Shapiro
  sequence, we derive a new substitution-based sequence
  with purely absolutely continuous diffraction spectrum.
\end{abstract}

\section{Introduction}

Substitution dynamical systems are widely used as toy models for
aperiodic order in one dimension~\cite{PF02}. The binary
Rudin--Shapiro sequence~\cite{{Rud59},{Sha51}} is a paradigm for a
substitution-based structure with (in its balanced weight case) purely
absolutely continuous diffraction spectrum; we refer to~\cite{BG13}
for details and background. Indeed, this deterministic sequence has
the stronger property that its two-point correlations vanish for
\emph{any} non-zero distance, as it would be the case for a random
sequence \cite{BG09}.  While currently only very few examples of
substitution-based sequences with this property are known, a
systematic generalisation in terms of Hadamard matrices \cite{Fra03}
allows the construction of one-dimensional as well as
higher-dimensional examples.

We start by reviewing Rudin's original construction~\cite{Rud59}, and
modify it to obtain a new example of a substitution-based structure
with absolutely continuous diffraction spectrum. Our approach differs
from the construction in~\cite{Fra03}. We then use recent work of
Bartlett~\cite{BG13} to investigate the properties of the new system
in more detail.

\section{Original construction of the Rudin--Shapiro sequence}

In 1958, Salem~\cite{Rud59} asked the following question: is it
possible to find a sequence $(\varepsilon_{n})_{n\in\N}\in\{\pm
1\}^{\N}$ such that there is a constant $C>0$ for which
\begin{equation}
\label{equation:0}
\sup_{\theta\in\mathbb{R}}\,\biggl|\sum_{n<N}\varepsilon_{n}
  \, e^{2\pi in\theta}\biggr|\, \leq\, C\sqrt{N}
\end{equation}
holds for any positive integer $N$? This is known as the `root $N$'
property, which implies absolutely continuous diffraction of the
sequence $(\varepsilon_{n})_{n\in\N}$ by an application of
\cite[Prop.~4.9]{MQ10}. Rudin~\cite{Rud59} and Shapiro~\cite{Sha51}
independently gave positive answers to the question by constructing
what is now known as the Rudin--Shapiro sequence. We briefly review
Rudin's original construction to obtain the substitution dynamical
system and its balanced weight version (with weights in $\pm 1$).

We start by defining polynomials $P_{k}(x)$ and $Q_{k}(x)$ of degree
$2^{k}$ for $k\in\mathbb{N}_{0}$ recursively by
\begin{equation}\label{equation:1}
\begin{split}
P_{k+1}(x)\, =\, P_{k}(x)+x^{2^{k}}Q_{k}(x),\\
Q_{k+1}(x)\, =\, P_{k}(x)-x^{2^{k}}Q_{k}(x),
\end{split}
\end{equation} 
with $P^{}_{0}(x)=Q^{}_{0}(x)=x$. Note that from
Eq.~\eqref{equation:1} it is clear that the first $2^k$ terms of
$P_{k+1}(x)$ and of $Q_{k+1}(x)$ coincide with those of $P_{k}(x)$,
and the remaining terms differ by a sign. By construction, $P_{k}(x)$
is of the form
\begin{equation}\label{equation:1b}
  P_{k}(x)\, =\,\sum_{n=1}^{2^{k}}\,\varepsilon_{n}\,x^n,
\end{equation}
so we can define a binary sequence
$(\varepsilon_{n})^{}_{n\in\mathbb{N}}\in\{\pm 1\}^{\mathbb{N}}$ from
the corresponding coefficients. This is the binary Rudin--Shapiro
sequence. For example,
$P_3(x)=x+x^{2}+x^{3}-x^{4}+x^{4}(x+x^{2}-x^{3}+x^{4})$, from which we
read off the sequence $111\overline{1}11\overline{1}1$ with
$\overline{1}=-1$.  The main ingredient in the proof of property
\eqref{equation:0} for this sequence is the parallelogram law,
\[
   |\alpha+\beta|^{2}+|\alpha-\beta|^{2}
   \, =\, 2|\alpha|^{2}+2|\beta|^{2},
\]
for $\alpha,\beta\in\C$; see~\cite{Rud59} for details. This can be
used to establish the bound \eqref{equation:0} on $P_{k}(e^{2\pi
  i\theta})$.

Often, the Rudin--Shapiro sequence is defined by a four-letter
substitution rule and a subsequent reduction map from four to two
letters. The underlying four-letter substitution on the alphabet
$\{A,B,C,D\}$ can be read off from the recursion \eqref{equation:1},
noting that the recursion implies the concatenation of the sequences
corresponding to $P_{k}$ and $Q_{k}$. Associating letters $A$ and $B$
to the coefficients in $P$ and $Q$ and the letters $C$ and $D$ to
those of $-Q$ and $-P$, respectively, this gives rise to the
four-letter substitution rule
\[
  \varrho\!:\quad A\mapsto AB, \quad B\mapsto AC, \quad C\mapsto DB, 
  \quad D\mapsto DC,
\]
which corresponds to the standard four-letter Rudin--Shapiro
subsitution.  Iterating the sequence on the initial seed $A$ gives
\[
   A\mapsto AB\mapsto ABAC\mapsto ABACABDB\mapsto \dots
\]
which, under the mapping $A\mapsto 1$, $B\mapsto 1$, $C\mapsto -1$,
$D\mapsto -1$ (which corresponds to the choice of signs above)
produces the binary Rudin--Shapiro sequence
$(\varepsilon_{n})^{}_{n\in\mathbb{N}}$.

\section{Generalising Rudin's construction}

Let us now modify Rudin's original argument by considering the
following system
\begin{equation}
\label{equation:2}
\begin{split}
  P_{k+1}(x)\, =\, P_{k}(x)+(-1)^{k}x^{2^{k}}Q_{k}(x),\\
  Q_{k+1}(x)\, =\, P_{k}(x)-(-1)^{k}x^{2^{k}}Q_{k}(x),
\end{split}
\end{equation}
starting again from $P^{}_{0}(x)=Q^{}_{0}(x)=x$. Since from now on we
shall only look at this set of recursion relations, we use the same
notation as above. What has changed in comparison with
Eq.~\eqref{equation:1} is that we swap the sign in the recursion
relation in each step. Since the sign does not affect the argument
used in Rudin's proof \cite{Rud59}, it is straightforward to show that
the parallelogram law can again be used to prove that the
corresponding \emph{new} sequence of coefficients
$(\varepsilon_{n})^{}_{n\in\mathbb{N}}$, defined as above via
Eq.~\eqref{equation:1b}, also satisfies the root $N$ property.

Can we read off the corresponding substitution rule as for the
Rudin--Shapiro case? Indeed this is possible, but it is slightly more
complicated, because the recursion relations in Eq.~\eqref{equation:2}
alternate. One way to deal with that is to look at two consecutive
steps together,
\begin{equation}
\label{equation:2b}
\begin{split}
  P_{k+2}(x)\, =\, P_{k}(x) + (-1)^{k}x^{2^{k}}Q_{k}(x) + 
                 (-1)^{k+1}x^{2\cdot 2^{k}} P_{k}(x)
                 + x^{3\cdot 2^{k}}Q_{k}(x),\\
  Q_{k+2}(x)\, =\, P_{k}(x) + (-1)^{k}x^{2^{k}}Q_{k}(x) - 
                 (-1)^{k+1}x^{2\cdot 2^{k}} P_{k}(x)
                 - x^{3\cdot 2^{k}}Q_{k}(x).
\end{split}
\end{equation}
Choosing $k$ to be even, say, and associating again four letters
$A,B,C,D$ to the sequences corresponding to $P,Q,-Q,-P$, we can read
off the substitution
\begin{equation}
\label{equation:3}
  \sigma\!:\quad 
  A\mapsto ABDB,\quad B\mapsto ABAC,\quad C\mapsto DCDB,\quad D\mapsto DCAC,
\end{equation} 
which is now a constant length substitution of length four, because we
used a double step of the recursion. Therefore Eq.~\eqref{equation:2b}
corresponds to concatenation of four sets of coefficients.  As before,
the binary sequence is obtained from iterating the substitution on the
initial letter $A$;
\begin{equation}\label{equation:3b}
   A\mapsto ABDB\mapsto ABDBABACDCACABAC\mapsto \cdots.
\end{equation}
By mapping $A$ and $B$ to $1$ and mapping $C$ and $D$ to $\overline{1}=-1$,
we obtain $11\overline{1}1111\overline{1}\overline{1}
\overline{1}1\overline{1}111\overline{1}\ldots$ as the initial part of
our new binary sequence.

Clearly, we can generate infinitely many such examples by changing the
signs in the original recursion relation \eqref{equation:1} in more
complicated ways. Any finite sequence of signs, when used
periodically, will give rise to a substitution-based system. However,
the length of the substitution will increase with the length of the
sign sequence. There neither appears to be an obvious relation between
the different Rudin--Shapiro type sequences obtained from this
construction, nor between these sequences and those derived from
Frank's construction \cite{Fra03}. In the latter case, the number of
letters in the alphabet increases with the size of the Hadamard
matrix, whereas all our sequences will only use four letters. Having
said that, it would be possible to use more letters rather than
consider multiple recursion steps, so the relation between these two
approaches still needs to be analysed in more detail.

\section{Spectral properties}

Since we have Rudin's original argument at our disposal, we know that
this new binary sequence satisfies the root $N$ property
\eqref{equation:0}, and hence has (in the balanced weight case with
weights $\pm 1$) purely absolutely continuous
diffraction. Nevertheless, we would like to use the remainder of this
paper to apply Bartlett's algorithm \cite{AB14} to our new
(four-letter substitution) sequence, and to verify that (in the
balanced weight case) it indeed has absolutely continuous diffraction
spectrum only. Due to space constraints, we cannot introduce all
quantities here, and instead refer to \cite{AB14} for definitions and
further details.

\begin{thm}
  The balanced weight sequence derived from the
  substitution rule\/ $\sigma$ of Eq.~\eqref{equation:3} has purely
  absolutely continuous diffraction spectrum.
\end{thm}
\begin{proof}
  The four instruction matrices $R_{i}$ and the substitution matrix
  $M$ can be read off of the substitution rule
  \eqref{equation:3}. They are given by
\[
R_0=
\begin{pmatrix}
1 & 1 & 0 & 0\\
0 & 0 & 0 & 0\\
0 & 0 & 0 & 0\\
0 & 0 & 1 & 1
\end{pmatrix},\quad
R_1=
\begin{pmatrix}
0 & 0 & 0 & 0\\
1 & 1 &0 &0\\
0& 0& 1&1\\
0& 0&0& 0
\end{pmatrix}, \quad
R_2=
\begin{pmatrix}
0& 1 & 0 &1 \\
0& 0&0&0\\
0&0&0&0\\
1&0&1&0
\end{pmatrix},\quad
R_3=
\begin{pmatrix}
0&0&0&0\\
1&0&1&0\\
0&1&0&1\\
0&0&0&0
\end{pmatrix},
\] 
with $M=R_{0}+R_{1}+R_{2}+R_{3}$. As $M^{2}$ has positive entries
only, the substitution is primitive. The second iterate of the seed
$A$ obtained in Eq.~\eqref{equation:3b} shows that the letter $A$ can
be preceded by either $B$ or $C$. Hence $A$ has two distinct
neighbourhoods and, by Pansiot's lemma~\cite[Lem.~3.6]{AB14}, the
substitution is aperiodic.

Since we are dealing with a constant-length substitution, the leading
(Perron--Frobenius) eigenvalue \cite[Thm.~2.2]{BG13} of $M$ is
$\lambda_{\textnormal{PF}}=4$ with (statistically normalised)
eigenvector and $u=\frac{1}{4}(1,1,1,1)$.  By
applying~\cite[Thm.~4.3]{AB14}, we obtain
$\widehat{\Sigma}(0)=\frac{1}{4}\sum_{\alpha}e_{\alpha\alpha}$, where
$e_{\alpha\alpha}$ denotes the unit vector corresponding to the word
$\alpha\alpha$. As $\sigma$ is a length four substitution, we have
$\Delta_{1}(1)=\{3\}$. Using~\cite[Thm.~4.3]{AB14}, we find
\[
  \widehat{\Sigma}(1)\,=\,
  \frac{1}{8}\left(0,1,1,0,1,0,0,1,1,0,0,1,0,1,1,0\right)\,=\,
  \widehat{\Sigma}(3),
\]
and 
\[
  \widehat{\Sigma}(2)\,=\,
  \frac{1}{8}\left(1,0,0,1,0,1,1,0,0,1,1,0,1,0,0,1\right)\,=\,
  \widehat{\Sigma}(4).
\]
One can then verify the following recursive relations
\begin{equation}
\label{equation:4}
   \widehat{\Sigma}(4n)\,=\,\widehat{\Sigma}(4),\quad 
   \widehat{\Sigma}(4n+1)\,=\,\widehat{\Sigma}(1),\quad 
   \widehat{\Sigma}(4n+2)\,=\,\widehat{\Sigma}(2), \quad
   \widehat{\Sigma}(4n+3)\,=\,\widehat{\Sigma}(3). 
\end{equation}
Using~\cite[Prop.~2.2]{AB14}, we calculate the ergodic decomposition
of the bi-substitution $\sigma\otimes\sigma$, and obtain
$E_{1}=\{AA,BB,CC,DD\}$ and $E_{2}=\{AD,DA,BC, CB\}$ as the two
ergodic classes and $T=\{AB,AC,BA,BD,CA,CD,DB,DC\}$ as the
transient part. From~\cite[Lem.~4.7]{AB14}, it then follows that 
\[
  v\, =\, w^{}_1
  \begin{pmatrix}
  1&0&0&0\\
  0&1&0&0\\
  0&0&1&0\\
  0&0&0&1\end{pmatrix}+w^{}_2
  \begin{pmatrix}
  0&0&0&1\\
  0&0&1&0\\
  0&1&0&0\\
  1&0&0&0
  \end{pmatrix}+\frac{w^{}_1+w^{}_2}{2}
  \begin{pmatrix}
  0&1&1&0\\
  1&0&0&1\\
  1&0&0&1\\
  0&1&1&0
  \end{pmatrix}.
\]
Diagonalising the matrix $v$, we obtain 
\[
  v^{}_{d}=
  \begin{pmatrix}
  2(w^{}_1+w^{}_2)&0&0&0\\
  0&w^{}_1-w^{}_2&0&0\\
  0&0&w^{}_1-w^{}_2&0\\
  0&0&0&0
  \end{pmatrix}.
\]
Setting $w^{}_1=1$, strong semi-positivity is equivalent to $w^{}_2$
satisfying $-1\leq w^{}_2\leq 1$. The extreme points are then given by
$(w^{}_1,w^{}_2)=(1,1)$ and $(w^{}_1,w^{}_2)=(1,-1)$. Thus the extreme
rays are
\[
  \begin{split}
  v_1=&(1,1,1,1,1,1,1,1,1,1,1,1,1,1,1,1),\\
  v_2=&(1,0,0,-1,0,1,-1,0,0,-1,1,0,-1,0,0,1).
  \end{split}
\]
As usual, $\lambda_{v^{}_1}=\delta^{}_{0}$, which gives rise to the
pure point component of the spectrum. Using Eq.~\eqref{equation:4},
one checks that $\widehat{{\lambda}_{v^{}_2}}(k)=0$ for all $k\neq 0$,
which gives us the absolutely continuous component. Thus, we have a
purely absolute continuous spectrum in the balanced weight case, in
which the pure point component is extinguished.
\end{proof}

\ack The authors would like to thank Ian Short for many helpful
comments. L.C.\ would like to thank the organising
committee for a young scientist award to attend ICQ13 in Kathmandu.


\section*{References}
\begin{thebibliography}{1}

\bibitem{PF02}
Pytheas~Fogg N 2002 \textit{Substitutions in Dynamics, Arithmetics and
  Combinatorics} (Lecture Notes in Mathematics vol 1794) (Springer,
  Berlin)

\bibitem{Rud59}
Rudin W 1959 Some theorems on Fourier coefficients
\textit{Proc.\ Amer.\ Math.\ Soc.} \textbf{10} 855--859

\bibitem{Sha51}
Shapiro H 1951 \textit{Extremal Problems for Polynomials and Power Series}
  Masters thesis (MIT, Boston)

\bibitem{BG13}
Baake M and Grimm U 2013 \textit{Aperiodic Order. Vol.~1. A Mathematical
  Invitation} (Encyclopedia of Mathematics and its Applications vol
  149) (Cambridge University Press, Cambridge)

\bibitem{BG09}
Baake M and Grimm U 2009 
Kinematic diffraction is insufficient to distinguish order 
from disorder
\textit{Phys.\ Rev.\ B} \textbf{79} 020203 and 
\textbf{80} 029903 (Erratum)

\bibitem{Fra03}
Frank N 2003 Substitution sequences in $\mathbb{Z}^{d}$ 
with a non-simple Lebesgue component in the spectrum
\textit{Ergod.\ Th.\ \& Dynam.\ Syst.} \textbf{23} 519--532

\bibitem{MQ10}
Queff{\'e}lec M 2010 \textit{Substitution Dynamical Systems\,---\,Spectral
  Analysis} 2nd ed (Lecture Notes in Mathematics vol 1294) (Springer,
  Berlin)

\bibitem{AB14}
Bartlett A 2014 Spectral theory of $\mathbb{Z}^{d}$ substitutions 
\textit{Ergod.\ Th.\ \& Dynam.\ Syst.} (to appear, preprint
 arXiv:1410.8106)

\end{thebibliography}
\end{document}